\documentclass[12pt]{article}
\usepackage{amsfonts,amssymb,amsthm,amsmath,cite,bbm}
\usepackage{mathptmx}
\usepackage{newtxtext}
\usepackage{enumerate}
\usepackage{enumitem}
\usepackage{hyperref}
\usepackage{comment}
\usepackage{tgpagella}

\usepackage{geometry}
\hypersetup{
	colorlinks,%
	citecolor=black,%
	filecolor=black,%
	linkcolor=black,%
	urlcolor=black
}

\newtheorem{theorem}{Theorem}[section]
\newtheorem{lemma}[theorem]{Lemma}

\newtheorem{corollary}[theorem]{Corollary}
\theoremstyle{remark}
\newtheorem{remark}[theorem]{Remark}
\numberwithin{equation}{section}

\newcommand{\N}{{\mathbb N}} 
\newcommand{\R}{{\mathbb R}}
\newcommand{\Rn}{{\mathbb R}^n}
\newcommand{\s}{\mathbb{S}}
\newcommand{\sn}{{\mathbb{S}^{n-1}}}

\newcommand{\K}{{\mathcal K}}
\newcommand{\Kn}{{\mathcal K}^n}

\newcommand{\hm}{\mathcal H}
\newcommand{\stirling}[2]{\genfrac{[}{]}{0pt}{}{#1}{#2}} 
\newcommand{\diffconv}{\mathbin{\diamond}} 

\newcommand{\fconvs}{{\mbox{\rm Conv}_{{\rm sc}}(\R^n)}} 
\newcommand{\fconvsH}{{\mbox{\rm Conv}_{{\rm sc}}(\R^{n-1})}} 
\newcommand{\fconvf}{{\mbox{\rm Conv}(\R^n; \R)}}
\newcommand{\proj}{\operatorname{proj}}
\newcommand{\gnom}{\operatorname{gno}} 

\newcommand{\infconv}{\mathbin{\Box}} 
\newcommand{\sq}{\mathbin{\vcenter{\hbox{\rule{.3ex}{.3ex}}}}} 

\newcommand{\MA}{\text{\rm MA}} 
\newcommand{\MAp}{\text{\rm MA}^{\!*}} 

\DeclareMathOperator{\oZ}{\operatorname{Z}}
\DeclareMathOperator{\oVb}{\overline{\operatorname{V}}}

\newcommand{\oZZb}[2]{\overline{\operatorname{V}}_{#1,#2}} 

\newcommand{\Grass}[2]{\operatorname{G}(#2,#1)}

\newcommand{\SO}{\operatorname{SO}}
\newcommand{\oO}{\operatorname{O}}

\renewcommand{\d}{\,\mathrm{d}}
\newcommand{\Hess}{{\operatorname{D}}^2}
\newcommand{\ind}{{\rm\bf I}}
\newcommand{\dom}{\operatorname{dom}} 
\newcommand{\epi}{\operatorname{epi}} 
\newcommand{\cR}{\operatorname{\mathcal{R}}} 

\title{Additive Kinematic Formulas\\for Convex Functions}
\author{Daniel Hug, Fabian Mussnig, and Jacopo Ulivelli}

\date{}

\begin{document}
\maketitle

\begin{abstract}
We prove a functional version of the additive kinematic formula as an application of the Hadwiger theorem on convex functions together with a Kubota-type formula for mixed Monge--Amp\`ere measures. As an application, we give a new explanation for the equivalence of the representations of functional intrinsic volumes as singular Hessian valuations and as integrals with respect to mixed Monge--Amp\`ere measures. In addition, we obtain a new integral geometric formula for mixed area measures of convex bodies, where integration on $\SO(n-1)\times \oO(1)$ is considered.
\bigskip

{\noindent
{\bf 2020 AMS subject classification:} 52A22 (26B25, 52A20, 52A39, 52A41, 52B45)\\
{\bf Keywords:} integral geometry, additive kinematic formula, convex function, valuation, mixed Monge--Amp\`ere measure, mixed area measure
}
\end{abstract}

\section{Introduction and Statement of Results}
\label{se:introduction}
For $n\in\N$ we denote by $\Kn$ the set of \emph{convex bodies} in $\Rn$, i.e., the set of non-empty, compact, convex subsets. Among the central objects in convex geometry are the \emph{intrinsic volumes} $V_j\colon\Kn\to \R$, $0\leq j\leq n$, which are given as coefficients in the Steiner formula
$$\operatorname{vol}_n(K+r B^n) = \sum_{j=0}^n r^{n-j} \kappa_{n-j} V_j(K)$$
for $r>0$ and $K\in\Kn$. Here, $\operatorname{vol}_n$ denotes the $n$-dimensional volume (i.e., the Lebesgue measure on $\Rn$), $B^n$ is the Euclidean unit ball in $\Rn$, and $\kappa_i$ denotes the $i$-dimensional volume of the unit ball in $\R^i$. Furthermore, for $\lambda,\mu\geq 0$ and $K,L\in\Kn$ we write
$$\lambda K+\mu L = \{\lambda x + \mu y : x\in K, y\in L\}$$
for the \emph{Minkowski sum} of the bodies $\lambda K$ and $\mu L$. The expression $2 V_{n-1}(K)$ gives the usual surface area of $K\in\Kn$ and if $\dim K\leq j$, i.e., if $K$ is contained in a $j$-dimensional affine subspace, then $V_j(K)$ is the usual $j$-dimensional volume of $K$ (we will thus use $V_j$ instead of $\operatorname{vol}_j$).

\medskip

Alternative but equivalent definitions of the intrinsic volumes can be given, for example, in terms of differential geometry (see \cite[(4.25), (4.26)]{Schneider_CB} or \cite[Theorem 4.9]{Hug_Weil_Lectures}) and integral geometry (see \cite[Remarks 5.1 and 5.5]{Hug_Weil_Lectures}). Another approach characterizes the operators $V_j$ by their unique properties, which we explain in the following. We call a map $\oZ\colon \Kn\to\R$ a \emph{valuation} if
$$\oZ(K\cup L)+\oZ(K\cap L)=\oZ(K)+\oZ(L)$$
for $K,L\in\Kn$ such that also $K\cup L\in\Kn$. The operator $\oZ$ is said to be \emph{translation invariant} if $\oZ(K+x)=\oZ(K)$ for $K\in\Kn$ and $x\in\Rn$ and it is \emph{rotation invariant} if $\oZ(\vartheta K)=\oZ(K)$ for $K\in\Kn$ and $\vartheta\in \SO(n)$. Here $\vartheta K = \{\vartheta x : x\in K\}$ and $\SO(n)$ denotes the special orthogonal group, i.e., the group of orientation preserving rotations of $\Rn$. The result below is due to Hadwiger \cite[Satz IV]{Hadwiger} and characterizes linear combinations of intrinsic volumes. Here and in the following, continuity of operators defined on $\Kn$ is understood with respect to the Hausdorff metric (see, for example, \cite[Section 1.8]{Schneider_CB} for details).

\begin{theorem}[Hadwiger's Theorem]
\label{thm:Hadwiger}
A map $\oZ\colon \Kn\to\R$ is a continuous, translation and rotation invariant valuation if and only if there exist $c_0,\ldots,c_n\in\R$ such that
$$\oZ(K)=\sum_{j=0}^n c_i V_i(K)$$
for $K\in\Kn$.
\end{theorem}
Among its numerous applications, the strength of Theorem~\ref{thm:Hadwiger} is particularly evident in integral geometry, where it provides almost effortless proofs of formulas that involve integration of geometric quantities with respect to invariant measures. See, for example, \cite{Hug_Schneider_Survey,Hug_Weil_Lectures,KlainRota}. One such result is the following additive kinematic formula for which we refer to \cite[Theorem 5.13]{Hug_Weil_Lectures} (see \cite[Corollary 5.2]{Hug_Weil_Lectures} for a more general local version and \cite[Theorem 4.4.6]{Schneider_CB} for a different approach).

\begin{theorem}[Additive Kinematic Formula]
\label{thm:additive_kinematic_formula}
For $0\leq j\leq n$ and $K,L\in\Kn$,
\begin{equation}
\label{eq:additive_kinematic_formula}
\int_{\SO(n)} V_j(K+\vartheta L)\d \vartheta = \sum_{k=0}^j \frac{\binom{2n-j}{n-j} \kappa_{n-k}\kappa_{n+k-j}}{\binom{2n-j}{n-k} \kappa_n\kappa_{n-j}} V_k(K) V_{j-k}(L),
\end{equation}
where $\d \vartheta$ denotes integration with respect to the Haar probability measure on $\SO(n)$.
\end{theorem}

\medskip

The aim of this article is to establish a functional version of Theorem~\ref{thm:additive_kinematic_formula}. For this, we denote by $\fconvf$ the set of convex functions $v\colon\Rn\to\R$. In \cite{Colesanti-Ludwig-Mussnig-5}, functional analogs of the intrinsic volumes on $\fconvf$ were introduced and characterized in a Hadwiger-type theorem. For $v\in \fconvf \cap C^2(\Rn)$ and $j \in \{0, \dots,n \}$, these functional intrinsic volumes are of the form
\begin{equation}
\label{eq:func_intr_vol_hess}
v\mapsto \int_{\Rn} \zeta(|x|) [\Hess v(x)]_j \d x,
\end{equation}
where $|x|$ denotes the Euclidean norm of $x\in\Rn$, $\d x$ denotes integration with respect to Lebesgue measure on $\R^n$, and $\zeta\colon (0,\infty)\to\R$ is continuous with bounded support with a possible singularity at $0^+$ (see Section~\ref{se:singular_hessian} for details). Here, $\Hess v(x)$ denotes the Hessian matrix of $v$ at $x\in\Rn$ and we write $[A]_j$ for the $j$th elementary symmetric function of the eigenvalues of the symmetric matrix $A\in\R^{n\times n}$.

While \eqref{eq:func_intr_vol_hess} is easy to understand, it turns out that this representation of functional intrinsic volumes as singular Hessian integrals is not well suited for an additive kinematic formula. It was shown in \cite[Theorem 2.5]{Colesanti-Ludwig-Mussnig-7} that \eqref{eq:func_intr_vol_hess} can be rewritten as
$$v\mapsto \int_{\Rn} \alpha(|x|) \d\MA(v[j],h_{B^n}[n-j];x),$$
where $\alpha$ is a continuous function with compact support on $[0,\infty)$ that is obtained from $\zeta$ via an integral transform (see Section~\ref{se:singular_hessian}). Here, we write $h_K(x)=\sup\{ \langle x,y\rangle \colon y\in K\}$, $x\in\Rn$, for the \emph{support function} of $K\in\Kn$, where $\langle \cdot\,,\cdot\rangle$ denotes the standard inner product on $\Rn$, and we remark that $h_{B^n}(x)=|x|$ for $x\in\Rn$. Moreover, $\MA(w_1,\ldots,w_n;\cdot)$ denotes the \emph{mixed Monge--Amp\`ere measure} of the functions $w_1,\ldots,w_n\in\fconvf$ and in the equation above the function $v$ is repeated $j$ times and $h_{B^n}$ is repeated $n-j$ times. Under additional $C^2$ assumptions on its arguments, the mixed Monge--Amp\`ere measure is absolutely continuous with respect to the Lebesgue measure and takes the form
$$\d\MA(w_1,\ldots,w_n;x)=\det(\Hess w_1(x),\ldots,\Hess w_n(x))\d x,$$
where $\det\colon (\R^{n\times n})^n\to\R$ denotes the mixed discriminant. For a more precise definition of this measure, we refer to Section~\ref{se:preliminaries}.

For our purposes, we will thus consider the (renormalized) \emph{functional intrinsic volumes} $\oZZb{j}{\alpha}^*\colon \fconvf\to\R$ given by
\begin{equation}
\label{eq:ozzb_ma}
\oZZb{j}{\alpha}^*(v)=\binom{n}{j}\frac{1}{\kappa_{n-j}}\int_{\Rn} \alpha(|x|)\d \MA(v[j],h_{B^n}[n-j];x)
\end{equation}
for $v\in\fconvf$, where $0\leq j\leq n$, and $\alpha\in C_c({[0,\infty)})$. This particular choice of normalization has the advantage that
$$\oZZb{j}{\alpha}^*(h_K)= \alpha(0) V_j(K)$$
for $K\in\Kn$ (see, for example, \cite[Lemma 4.6]{Hug_Mussnig_Ulivelli_supports}). The Hadwiger theorem on the space $\fconvf$, which was first established in \cite[Theorem 1.5]{Colesanti-Ludwig-Mussnig-5}, is the following result. For the version stated below, see \cite[Theorem 2.6]{Colesanti-Ludwig-Mussnig-7}. See also \cite[Theorem 1.1]{Colesanti-Ludwig-Mussnig-8} and \cite[Theorem 1.2]{Knoerr_singular}.

For the statement of the result, we recall some terminology. Continuity of a functional $\oZ\colon \fconvf\to\R$
 is understood with respect to epi-convergence, which on $\fconvf$ is equivalent to pointwise convergence. The operator $\oZ$ is \emph{dually epi-translation invariant} if $\oZ(v+f)=\oZ(v)$ for $v\in\fconvf$ and affine functions $f$ on $\Rn$ and it is \emph{rotation invariant} if $\oZ(v\circ \vartheta^{-1})=\oZ(v)$ for $v\in\fconvf$ and $\vartheta\in \SO(n)$. Lastly, $\oZ$ is a \emph{valuation} if
 \begin{equation}
 \label{eq:def_val_func}
 \oZ(v\wedge w)+\oZ(v \vee w)=\oZ(v)+\oZ(w)
 \end{equation}
for $v,w\in\fconvf$ such that also $v\wedge w\in\fconvf$, where $v\wedge w$ and $v\vee w$ denote the pointwise minimum and maximum of $v$ and $w$, respectively.

\begin{theorem}
\label{thm:hadwiger_fconvf}
A functional $\oZ\colon \fconvf\to\R$ is a continuous, dually epi-translation and rotation invariant valuation if and only if there exist functions $\alpha_0,\ldots,\alpha_n\in C_c({[0,\infty)})$ such that
$$\oZ(v)=\sum_{j=0}^n \oZZb{j}{\alpha_j}^*(v)$$
for $v\in\fconvf$.
\end{theorem}

We will use the functional Hadwiger theorem together with a Kubota-type formula for (conjugate) mixed Monge--Amp\`ere measures (see Lemma~\ref{le:ck_map}) to prove the following functional counterpart of Theorem~\ref{thm:additive_kinematic_formula}.

\begin{theorem}
	\label{thm:add_kin_form_ma}
	If $0\leq j\leq n$ and $\alpha\colon [0,\infty)\to [0,\infty)$ is measurable, then
	\begin{align}
	    \begin{split}
	        \label{eq:add_kin_form_ma}
	   \kappa_n \int_{\SO(n)} \int_{\Rn} \alpha(|y|) \d&\MA(v + (w\circ \vartheta^{-1})[j],h_{B^n}[n-j];y) \d\vartheta\\
	=\sum_{i=0}^j \binom{j}{i} \int_{\Rn}\int_{\Rn}& \alpha(\max\{|x|,|y|\})\\
 &\d\MA(w[j-i],h_{B^n}[n-j+i];y)\d\MA(v[i],h_{B^n}[n-i];x)
	    \end{split}
	\end{align}
	for $v,w\in\fconvf$.
\end{theorem}
Observe that the left side of \eqref{eq:add_kin_form_ma} can be rewritten as a multiple of
$$\int_{\SO(n)} \oZZb{j}{\alpha}^*(v+(w\circ \vartheta^{-1})) \d\vartheta,$$
which resembles \eqref{eq:additive_kinematic_formula}. However, in general, the right side of \eqref{eq:add_kin_form_ma} is not a sum of products of functional intrinsic volumes. A case in which this is possible is given by Corollary~\ref{cor:add_kin_ind} below.
In Section~\ref{se:further_results} we also show how Theorem~\ref{thm:additive_kinematic_formula} can be retrieved from Theorem~\ref{thm:add_kin_form_ma} and treat further consequences, such as formulas for functional analogs of mixed volumes (Corollary~\ref{cor:func_mix}) or analytic versions of the Minkowski difference (Corollary~\ref{cor:diffconv}).

As a further application of Theorem~\ref{thm:add_kin_form_ma}, we establish a novel explanation of the aforementioned equivalence between \eqref{eq:ozzb_ma} and the singular Hessian integrals \eqref{eq:func_intr_vol_hess} in Section~\ref{se:singular_hessian}.

\medskip

Lastly, in Section~\ref{se:formulas_bodies} we study the implications of Theorem~\ref{thm:add_kin_form_ma} for mixed area measures of convex bodies. We write $S_{n-1}(K,\cdot)$ for the \emph{surface area measure} of $K\in\Kn$, which is a Borel measure on the unit sphere $\sn$. For a body $K$ of dimension $n$ and a Borel set $\omega\subseteq \sn$, the expression $S_{n-1}(K,\omega)$ gives the ${(n-1)}$-dimensional Hausdorff measure, denoted by $\hm^{n-1}$, of all boundary points $x\in \partial K$ at which $K$ has an outer unit normal in $\omega$ (we refer to \cite[Section 4.2]{Schneider_CB} for a detailed description). The coefficients $S(K_{i_1},\ldots,K_{i_{n-1}},\cdot)$ in the polynomial expansion
\begin{equation*}
S_{n-1}(\lambda_1 K_1+\cdots+\lambda_m K_m,\cdot)=\sum_{i_1,\ldots,i_{n-1}=1}^m \lambda_{i_1}\cdots \lambda_{i_{n-1}} S(K_{i_1},\ldots, K_{i_{n-1}},\cdot)
\end{equation*}
for $m\in\N$, $\lambda_1,\ldots,\lambda_m\geq 0$, and $K_1,\ldots,K_m\in\Kn$, are the \emph{mixed area measures} of the bodies $K_{i_1},\ldots,K_{i_{n-1}}$. For $0 \leq j \leq n-1$ we consider measures of the form $S(K [j], B_H^{n-1}[n-1-j],\cdot)$, where the body $K\in\Kn$ is repeated $j$ times and the $(n-1)$-dimensional unit ball
$$B_H^{n-1}=B^n \cap e_n^\perp=\{(x_1,\ldots,x_n)\in \Rn : x_1^2+\cdots +x_{n-1}^2 \leq 1, x_n=0\}$$
is repeated $(n-1-j)$ times. Here, $e_n$ denotes the $n$th basis vector of the standard orthonormal basis of $\Rn$ and we write $H=e_n^\perp$. The authors studied these measures and their connection with mixed Monge--Amp\`ere measures and functional intrinsic volumes in more detail in \cite{Hug_Mussnig_Ulivelli_supports}.

For $n\ge 2$, we identify $\SO(n-1)$ as the group of rotations that fix $e_n$ and $\oO(1)$ as the group that consists of the identity and $\operatorname{diag}(1,\ldots,1,-1)$. We use Theorem~\ref{thm:add_kin_form_ma} to prove the following result, where integration on $\SO(n-1)\times \oO(1)$ is with respect to the Haar probability measure. In addition, we write $z_n=\langle z,e_n \rangle$, for $z\in\sn$.

\begin{theorem}
\label{thm:formula_mx_area_meas}
Let $n\geq 2$. If $0\leq j\leq n-1$ and $\beta \colon [0,1]\to [0,\infty)$ is measurable, then
\begin{align}\label{eq:thm_kin_bod}
    &\int_{\SO(n-1)\times \oO(1)} \int_{\sn} |z_n| \beta(|z_n|) \d S((K+\eta L) [j],B_H^{n-1}[n-1-j],z) \d\eta \notag\\
&\quad= \frac{1}{2\kappa_{n-1}} \sum_{i=0}^j \binom{j}{i} \int_{\sn}\int_{\sn} |w_n| |z_n| \beta(\min\{|w_n|,|z_n|\})\notag\\
&\qquad\qquad \d S(L [j-i],B_H^{n-1}[n-1-j+i],w) \d S(K [i],B_H^{n-1}[n-1-i],z)
\end{align}
for $K,L\in\Kn$.
\end{theorem}

A rotational integral formula for mixed area measures which is equivalent to Theorem~\ref{thm:formula_mx_area_meas} is provided in Corollary~\ref{corformula_mx_area_meas}. 

\section{Preliminaries}
\label{se:preliminaries}

Throughout this section, we state some results on convex functions. For general references, we refer to\cite{Rockafellar,RockafellarWets,Schneider_CB}.

For $v\in\fconvf$ we write $\partial v(x)$ for the \emph{subdifferential} of $v$ at $x\in\Rn$, which is the set
$$\partial v(x)= \{y\in\Rn : v(z)\geq v(x)+\langle y,z-x\rangle \; \forall z\in\Rn\}.$$
The function $v$ is differentiable at $x$ if and only if $\partial v(x)$ contains only one element, namely the gradient $\nabla v(x)$.

The \emph{Monge--Amp\`ere measure} of $v$, which is a Radon measure on $\Rn$, is defined as
$$\MA(v;B)= V_n\left(\bigcup_{b\in B}\partial v(b)\right)$$
for Borel sets $B\subseteq \Rn$ (see, for example, \cite[Theorem 2.3]{Figalli_MA}). The \emph{mixed Monge--Amp\`ere measure}, which is associated to an $n$-tuple of elements of $\fconvf$, is now given by the relation
\begin{equation}
\label{eq:def_mixed_ma}
\MA(\lambda_1 v_1+\cdots + \lambda_m v_m;\cdot) = \sum_{i_1,\ldots,i_n=1}^m \lambda_{i_1}\cdots \lambda_{i_n} \MA(v_{i_1},\ldots,v_{i_n};\cdot),
\end{equation}
where $m\in\N$, $v_1,\ldots,v_m\in\fconvf$, and $\lambda_1,\ldots,\lambda_m\geq 0$. Equation \eqref{eq:def_mixed_ma} uniquely determines the mixed Monge--Amp\`ere measure if we additionally assume that it is symmetric in its entries. See \cite{Trudinger_Wang_Hessian_III} and \cite[Theorem 4.3]{Colesanti-Ludwig-Mussnig-7}.

\medskip

For a convex function $w\colon \Rn\to(-\infty,\infty]$, we consider its \emph{convex conjugate} or \emph{Legendre--Fenchel transform}
$$w^*(x)=\sup\left\{ \langle x,y\rangle - w(y):  y\in\Rn \right\}$$
for $x\in\Rn$. For each $v\in\fconvf$, the convex conjugate $v^*$ is a lower semicontinuous, convex function on $\Rn$ with values in $(-\infty,\infty]$, which satisfies $v(\bar{x})<\infty$ for at least one $\bar{x}\in\Rn$ and which is \emph{super-coercive}, that is,
$$\lim_{|x|\to\infty} \frac{v^*(x)}{|x|}=\infty.$$
We denote the set of all such functions by $\fconvs$ and remark this duality can be stated as $u^*\in\fconvf$ if and only if $u\in\fconvs$.

While the space of convex bodies is naturally embedded into $\fconvf$ by associating with each body $K\in\Kn$ its support function $h_K\in\fconvf$, the canonical representative of $K$ in $\fconvs$ is given by its \emph{convex indicator function}
\begin{equation}
\label{eq:ind_sup_conj}
\ind_K(x)=h_K^*(x)=\begin{cases}0\quad&\text{if } x\in K,\\ \infty\quad&\text{else.} \end{cases}
\end{equation}

We equip the space $\fconvs$ with the topology associated with epi-convergence, where a sequence of convex functions $w_j\colon \Rn\to(-\infty,\infty]$, $j\in\N$, epi-converges to $w\colon \Rn\to(-\infty,\infty]$ if for every $x\in\Rn$,
\begin{itemize}
    \item $w(x)\leq \liminf_{j\to \infty} w_j(x_j)$ for every sequence $x_j\to x$ and
    \item $w(x)=\lim_{j\to\infty} w_j(x_j)$ for some sequence $x_j\to x$.
\end{itemize}
By \cite[Theorem 11.34]{RockafellarWets}, convex conjugation is a homeomorphism between $\fconvf$ and $\fconvs$. Let us remark that while on $\fconvf$ epi-convergence coincides with pointwise convergence, this is not the case anymore on $\fconvs$. For an alternative description of epi-convergence on $\fconvs$ which uses Hausdorff convergence of level sets, we refer to \cite[Lemma 5]{Colesanti-Ludwig-Mussnig-1}.

We need the following result, which is a consequence of \cite[Lemma 3.3]{Colesanti-Ludwig-Mussnig-6}.
\begin{lemma}
\label{le:joint_cont}
The map
$$(\vartheta,u)\mapsto u\circ \vartheta$$
is jointly continuous on $\SO(n)\times \fconvs$.
\end{lemma}

For a convex function $w\colon \Rn\to(-\infty,\infty]$ we write
$$\epi(w)=\{(x,t)\in\Rn\times\R : w(x)\leq t\}$$
for its \emph{epi-graph}, which is a convex subset of $\Rn\times \R$. For the convex conjugate of the pointwise sum of two functions $v_1$ and $v_2$ in $\fconvf$ we have
\begin{equation}
\label{eq:epi_sum}
\epi (v_1+v_2)^*=\epi(v_1^*) + \epi(v_2^*).
\end{equation}
The corresponding operation on $\fconvs$ is \emph{infimal convolution} or \emph{epi-sum} which for $u_1,u_2\in\fconvs$ is given by
$$(u_1\infconv u_2) (x)=\inf\left\{u_1(x-y)+u_2(y):y\in\Rn \right\}$$
for $x\in\Rn$. The set in \eqref{eq:epi_sum} can now be written as $\epi(v_1^*\infconv v_2^*)$. By the preceding exposition, the following result, which can be found in \cite[Theorem 7.46 (a)]{RockafellarWets}, is easy to see.
\begin{lemma}
\label{le:infconv_cont}
Let $u_j,v_j\in\fconvs$ for $j\in\N$. If $u_j$ epi-converges to $u\in\fconvs$ and $v_j$ epi-converges to $v\in\fconvs$, then $u_j\infconv v_j$ epi-converges to $u\infconv v$.
\end{lemma}
Next, for the convex conjugate of the pointwise multiplication of $v\in\fconvf$ with $\lambda> 0$ we have
$$\epi(\lambda v)^* = \lambda \epi(v^*) = \epi(\lambda \sq v^*),$$
where $(\lambda \sq u)(x)=\lambda u(\frac x\lambda)$ denotes the \emph{epi-multiplication} of $u\in\fconvs$ with $\lambda>0$. This operation continuously extends to $\lambda=0$ with $0\sq u = \ind_{\{o\}}$.

\medskip

The \emph{conjugate Monge--Amp\`ere measure} of $u\in\fconvs$ is defined by
$$\MA^*(u;\cdot)=\MA(u^*;\cdot)$$
or equivalently
$$\int_{\Rn} \beta(y) \d\MA^*(u;y)=\int_{\dom(u)} \beta(\nabla u(x))\d x$$
for measurable $\beta\colon\Rn\to[0,\infty)$. Here,
$$\dom(u)=\{x\in\Rn : u(x)<\infty\}$$
is the \emph{domain} of $u$ and it follows from Rademacher's theorem that a convex function is differentiable almost everywhere (w.r.t.~the Lebesgue measure) on its domain. Similarly, for $u_1,\ldots,u_n\in\fconvs$ the \emph{conjugate mixed Monge--Amp\`ere measure} is given by
\begin{equation}
\label{eq:def_mixed_map}
\MAp(u_1,\ldots,u_n;\cdot)=\MA(u_1^*,\ldots,u_n^*;\cdot)
\end{equation}
and satisfies
\begin{equation}
\label{eq:cma_poly_exp}
\MAp\big((\lambda_1\sq u_1)\infconv \cdots \infconv (\lambda_m\sq u_m);\cdot\big)=\sum_{i_1,\ldots,i_n=1}^m \lambda_{i_1}\cdots \lambda_{i_n} \MAp(u_{i_1},\ldots,u_{i_n};\cdot)
\end{equation}
for $m\in\N$, $u_1,\ldots,u_m\in\fconvs$, and $\lambda_1,\ldots,\lambda_m\geq 0$.

We use \eqref{eq:ind_sup_conj} and \eqref{eq:def_mixed_map} to obtain the following equivalent formulation of \cite[Lemma 4.6]{Hug_Mussnig_Ulivelli_supports}. Here we write $V\colon (\Kn)^n\to\R$ for the \emph{mixed volume}, which is defined as the unique symmetric map such that
$$V_n(\lambda_1 K_1 + \cdots + \lambda_m K_m) = \sum_{i_1,\ldots,i_n=1}^m \lambda_{i_1} \cdots \lambda_{i_n} V(K_{i_1},\ldots,K_{i_n})$$
for $m\in\N$, $K_1,\ldots,K_m\in\Kn$, and $\lambda_1,\ldots,\lambda_m\geq 0$. See \cite[Theorem 5.1.7]{Schneider_CB} for further details on mixed volumes. 

\begin{lemma}
\label{le:MAp_bodies}
If $K_1,\ldots,K_n\in\Kn$, then
$$\MAp(\ind_{K_1},\ldots,\ind_{K_n};B)= V(K_1,\ldots,K_n) \delta_o(B)$$
for Borel sets $B\subseteq \Rn$, where $\delta_o$ is the Dirac measure at the origin. In particular,
$$\binom{n}{j} \MAp(\ind_K[j],\ind_{B^n}[n-j];B)=\kappa_{n-j} V_j(K) \delta_o(B)$$ 
for $0\leq j\leq n$.
\end{lemma}

The next result is a Kubota-type formula for conjugate mixed Monge--Amp\`ere measures and was established by the authors in \cite[Theorem 5.1]{Hug_Mussnig_Ulivelli_supports}. Here, for $1\leq k\leq n$ we denote by $\Grass{k}{n}$ the Grassmannian of $k$-dimensional linear subspaces of $\Rn$ and integration on this space is always understood with respect to the Haar probability measure. For $u\in\fconvs$ and $E\in\Grass{k}{n}$ we write
$$\proj_E u(x_E) = \min\nolimits_{y\in E^\perp} u(x_E +y)$$
with $x_E\in E$, for the \emph{projection function} of $u$.

\begin{lemma}
\label{le:ck_map}
If $1\leq k <n$ and $\varphi\colon \Rn\to [0,\infty)$ is measurable, then
\begin{multline*}
\frac{1}{\kappa_n} \int_{\Rn} \varphi(y) \d\MAp(u_1,\ldots,u_k,\ind_{B^n}[n-k];y)\\
=\frac{1}{\kappa_k} \int_{\Grass{k}{n}}\int_E \varphi(y_E) \d\MAp_E(\proj_E u_1,\ldots,\proj_E u_k;y_E) \d E
\end{multline*}
for $u_1,\ldots,u_k\in\fconvs$.
\end{lemma}

For $t\geq 0$, let $u_t\in \fconvs$ be defined by $u_t(x)=t\vert x\vert + \ind_{B^n}(x)$ for $x\in\Rn$. We need the following result, which is a consequence of \cite[Lemma 8.4]{Colesanti-Ludwig-Mussnig-7} together with the defining relation \eqref{eq:def_mixed_map}.

\begin{lemma}
\label{le:int_map_ut}
If $1\leq j \leq n$ and $\alpha\in C_c([0,\infty))$, then
$$\int_{\Rn}\alpha(|y|) \d\MAp(u_t[j],\ind_{B^n}[n-j];y) = \kappa_n \alpha(t)$$
for $t\geq 0$.
\end{lemma}

\medskip

Lastly, we need some results on valuations on $\fconvs$, which are defined analogously to \eqref{eq:def_val_func}. By \cite[Proposition 3.5]{Colesanti-Ludwig-Mussnig-4}, a map $\oZ\colon \fconvs \to \R$ is a valuation if and only if $v\mapsto \oZ^*(v)=\oZ(v^*)$ is a valuation on $\fconvf$. We say that $\oZ$ is \emph{epi-translation invariant} if $\oZ^*$ is dually epi-translation invariant or equivalently if $\oZ(u\circ \tau^{-1}+c)=\oZ(u)$ for $u\in\fconvs$, translations $\tau$ on $\Rn$, and $c\in\R$. The operator $\oZ$ is \emph{epi-homogeneous} of degree $j\in\N$ if $\oZ(\lambda \sq u)=\lambda^j \oZ(u)$ for $u\in\fconvs$ and $\lambda\geq 0$.

The following result from \cite[Proposition 5.3]{Colesanti-Ludwig-Mussnig-7} provides some examples of valuations on $\fconvs$. We denote by $C_c(\Rn)$ the set of continuous real-valued functions with compact support on $\Rn$.

\begin{lemma}
	\label{le:mixed_val}
	Let $\varphi\in C_c(\Rn)$ and $0\leq j\leq n$. If $u_1,\ldots,u_{n-1}\in\fconvs$, then
	$$u\mapsto \int_{\Rn} \varphi(x) \d\MAp(u_1,\ldots,u_{n-j},u[j];x)$$
	defines a continuous, epi-translation invariant valuation on $\fconvs$ that is epi-homogeneous of degree $j$.
\end{lemma}

We say that a map $\oZ$ on $\fconvs$ is \emph{rotation invariant} if $\oZ(u\circ \vartheta^{-1})=\oZ(u)$ for $u\in\fconvs$ and $\vartheta\in\SO(n)$. The following Hadwiger-type result, provided in \cite[Theorem 1.3]{Colesanti-Ludwig-Mussnig-5}, is equivalent to Theorem~\ref{thm:hadwiger_fconvf} and shows that not many examples of valuations remain under the additional assumption of rotation invariance. For the version stated below, see \cite[Theorem 1.7]{Colesanti-Ludwig-Mussnig-7}.

\begin{theorem}
\label{thm:hadwiger_fconvs}
A functional $\oZ\colon \fconvs\to\R$ is a continuous, epi-translation and rotation invariant valuation if and only if there exist functions $\alpha_0,\ldots,\alpha_n\in C_c({[0,\infty)})$ such that
$$\oZ(u)=\sum_{j=0}^n \int_{\Rn} \alpha_j(|y|)\d\MAp(u[j],\ind_{B^n}[n-j];y)$$
for $u\in\fconvs$.
\end{theorem}

Let us remark that Lemma~\ref{le:int_map_ut} shows that the operator $\oZ$ that appears in Theorem~\ref{thm:hadwiger_fconvs} uniquely determines the densities $\alpha_j$.

\section{Proof of Theorem~\ref{thm:add_kin_form_ma}}
Throughout this section, we use the abbreviated notation
$$\MAp_j(u;\cdot) = \MAp(u[j],\ind_{B^n}[n-j];\cdot)=\MA(u^*[j],h_{B^n}[n-j];\cdot)$$
for $0\leq j\leq n$ and $u\in\fconvs$, which was introduced in \cite{Colesanti-Ludwig-Mussnig-7}.

To prove the next result, we follow the strategy of the proof of \cite[Lemma 3.4]{Colesanti-Ludwig-Mussnig-6}.  Moreover, we use that $(u\circ\vartheta)^*=u^*\circ \vartheta$, for $u\in \fconvs$ and $\vartheta\in \SO(n)$, and $\MA(v;\vartheta B)=\MA(v\circ\vartheta;B)$, for $v\in \fconvf$, $\vartheta\in \SO(n)$, and Borel sets $B \subseteq\R^n$.

\begin{lemma}
\label{le:int_son_cont}
For any $1\leq j\leq n$, fixed $\bar{v}\in\fconvs$ and $\alpha\in C_c([0,\infty))$ the map
\begin{equation}
\label{eq:u_mapsto_int}
u\mapsto \int_{\SO(n)} \int_{\Rn} \alpha(|y|) \d\MAp_j\big(u\infconv (\bar{v}\circ\vartheta^{-1});y\big) \d\vartheta
\end{equation}
is a continuous, epi-translation and rotation invariant valuation on $\fconvs$.
\end{lemma}
\begin{proof}
We start by showing that
\begin{equation}
\label{eq:joint_cont_int}
(\vartheta,u)\mapsto \int_{\Rn} \alpha(|y|) \d\MAp_j\big(u\infconv (\bar{v}\circ\vartheta^{-1});y\big)
\end{equation}
is jointly continuous on $\SO(n)\times \fconvs$. For this, observe that it follows from the rotational symmetry of the integrand and the two basic facts mentioned before the statement of the lemma that
\begin{align*}
\int_{\Rn} \alpha(|y|) \d\MAp_j\big(u\infconv (\bar{v}\circ\vartheta^{-1});y\big) &= \int_{\Rn} \alpha(|y|) \d\MAp_j\big((u\infconv (\bar{v}\circ\vartheta^{-1}))\circ \vartheta ;y\big)\\
&= \int_{\Rn} \alpha(|y|) \d\MAp_j\big((u\circ \vartheta)\infconv \bar{v} ;y\big)
\end{align*}
for every $u\in\fconvs$ and $\vartheta\in\SO(n)$. Thus, by Lemma~\ref{le:infconv_cont}, Lemma~\ref{le:mixed_val}, and Lemma~\ref{le:joint_cont}, 
the map given in \eqref{eq:joint_cont_int} is jointly continuous on $\SO(n)\times \fconvs$.

Let $u_i\in\fconvs$, $i\in\N$, be such that $u_i$ epi-converges to some $\bar{u}\in\fconvs$, which means that $\{u_i : i\in\N\} \cup \{\bar{u}\}$ is sequentially compact. Together with the fact that $\SO(n)$ is compact and the map given in \eqref{eq:joint_cont_int} is jointly continuous, it follows that the supremum
\begin{equation}\label{eq:supfinite}
\sup\left\{\left\vert \int_{\Rn} \alpha(|y|) \d\MAp_j\big(u_i\infconv (\bar{v}\circ\vartheta^{-1});y\big) \right \vert: i\in\N, \vartheta \in \SO(n) \right\}
\end{equation}
is finite. Thus, we may apply the dominated convergence theorem to obtain
\begin{multline*}
\lim_{i\to\infty} \int_{\SO(n)} \int_{\Rn} \alpha(|y|) \d\MAp_j\big(u_i\infconv (\bar{v}\circ\vartheta^{-1});y\big) \d\vartheta\\
= \int_{\SO(n)} \int_{\Rn} \alpha(|y|) \d\MAp_j\big(\bar{u}\infconv (\bar{v}\circ\vartheta^{-1});y\big) \d\vartheta,
\end{multline*}
which shows that \eqref{eq:u_mapsto_int} is continuous.

Lastly, it follows from the properties of infimal convolution and the corresponding properties provided in Lemma~\ref{le:mixed_val} that \eqref{eq:u_mapsto_int} defines an epi-translation and rotation invariant valuation.
\end{proof}

\begin{remark}
    An alternative argument showing that the supremum in \eqref{eq:supfinite} is finite, can be based on \cite[Remark 5.2]{Hug_Mussnig_Ulivelli_supports}.
\end{remark}

For the proof of the main result of this section, we need the elementary property
\begin{equation}
\label{eq:map_bn}
\MAp(\ind_{B^n};\cdot) = \kappa_n \delta_{o}
\end{equation}
which is a special case of Lemma~\ref{le:MAp_bodies}. The next result is the equivalent version of Theorem~\ref{thm:add_kin_form_ma} on $\fconvs$.

\begin{theorem}
	\label{thm:add_kin_form_map}
	If $0\leq j\leq n$ and $\alpha\colon [0,\infty)\to[0,\infty)$ is measurable, then
	\begin{multline*}
	\kappa_n \int_{\SO(n)} \int_{\Rn} \alpha(|y|) \d\MAp_j\big(u\infconv(v\circ \vartheta^{-1});y\big) \d\vartheta\\
	=\sum_{i=0}^j \binom{j}{i} \int_{\Rn}\int_{\Rn} \alpha(\max\{|x|,|y|\})\d\MAp_{j-i}(v;y)\d\MAp_i(u;x)
	\end{multline*}
	for $u,v\in\fconvs$.
\end{theorem}
\begin{proof}
For the proof, it is sufficient to consider a function $\alpha\in C_c({[0,\infty)})$. 
	If $j=0$, the statement trivially follows from \eqref{eq:map_bn}. Thus, we will assume $1\leq j \leq n$ throughout the following. For $u,v\in\fconvs$ set
	\begin{equation}
	\label{eq:def_oz_u_v}
	\oZ(u,v)=\int_{\SO(n)} \int_{\Rn} \alpha(|y|) \d\MAp_j\big(u\infconv(v\circ \vartheta^{-1});y\big) \d\vartheta.
	\end{equation}
	For fixed $\bar{v}\in\fconvs$ it follows from Lemma~\ref{le:int_son_cont} that $u\mapsto \oZ(u,\bar{v})$ is a continuous, epi-translation and rotation invariant valuation on $\fconvs$. Thus, by Theorem~\ref{thm:hadwiger_fconvs} there exist functions $\alpha_{i,\bar{v}}\in C_c([0,\infty))$, $0\leq i\leq n$, such that
	$$\oZ(u,\bar{v})= \sum_{i=0}^n \int_{\Rn} \alpha_{i,\bar{v}}(|y|)\d\MAp_i(u;y)$$
	for every $u\in\fconvs$. In particular, if we choose $u=\lambda\sq u_t$, where $u_t=t|\cdot|+\ind_{B^n}$, it follows from \eqref{eq:map_bn}, Lemma~\ref{le:mixed_val}, and Lemma~\ref{le:int_map_ut} that
	\begin{equation}
	\label{eq:oz_lambda_u_t_bar_v}
	\oZ(\lambda \sq u_t,\bar{v}) = \kappa_n \left(\alpha_{0,\bar{v}}(0) + \sum_{i=1}^n  \lambda^i \alpha_{i,\bar{v}}(t)\right)
	\end{equation}
	for every $t,\lambda\geq 0$. Since for every fixed $\bar{u}\in\fconvs$ also the map $v\mapsto \oZ(\bar{u},v)$ is a continuous, epi-translation and rotation invariant valuation, it follows from \eqref{eq:oz_lambda_u_t_bar_v} together with homogeneity that for every fixed $\bar{t}\geq 0$ each of the maps 
	$$v\mapsto \alpha_{0,{v}}(0)\quad\text{and}\quad v\mapsto \alpha_{i,v}(\bar{t}),\;1\leq i\leq n,$$
	is a continuous, epi-translation and rotation invariant valuation. Thus, by Theorem~\ref{thm:hadwiger_fconvs} there exist functions $\alpha_{0,k}\in C_c([0,\infty)), \alpha_{i,k}(\bar{t},\cdot)\in C_c([0,\infty))$, $1\leq i\leq n$, $0\leq k\leq n$, such that
	$$\alpha_{0,v}(0)=\sum_{k=0}^n \int_{\Rn} \alpha_{0,k}(|y|) \d\MAp_k(v;y)$$
	and
	$$\alpha_{i,v}(\bar{t}) =\sum_{k=0}^n \int_{\Rn} \alpha_{i,k}(\bar{t},|y|)\d\MAp_k(v;y)$$
	for every $v\in\fconvs$ and $1\leq i\leq n$. We thus have
	\begin{align}
	\begin{split}
	\label{eq:oz_alpha_ik}
	\oZ(u,v)&= \kappa_n \alpha_{0,v}(0)+ \sum_{i=1}^n \int_{\Rn} \alpha_{i,v}(|x|)\d\MAp_i(u;x)\\
	&= \kappa_n \sum_{k=0}^n \int_{\Rn}\alpha_{0,k}(|y|)\d\MAp_k(v;y)\\
	&\qquad +\sum_{i=1}^n\sum_{k=0}^n \int_{\Rn} \int_{\Rn} \alpha_{i,k}(|x|,|y|) \d\MAp_k(v;y)\d\MAp_i(u;x)
	\end{split}
	\end{align}
	for every $u,v\in\fconvs$.
	
	It remains to determine the relation between $\alpha_{i,k}$ and $\alpha$. In order to do this, we will evaluate $\oZ(u,v)$ at $u=\lambda\sq u_s$ and $v=\mu \sq u_t$ with $\lambda,\mu,s,t\geq 0$. Notice that \[\lambda\sq u_s(x)=s|x|+\ind_{\lambda B^n}(x) \quad \text{and} \quad (\mu\sq u_t)\circ \vartheta^{-1}(x)=t|x|+\ind_{\mu B^n}(x) \] 
 for  $\vartheta\in \SO(n)$ and $x\in\R^n$.
 For $0\le s\le t$ we now have 
 \begin{align*}
\left((\lambda\sq u_s)\infconv \big((\mu \sq u_t)\circ \vartheta^{-1}\big)\right) (x)
     &=\inf\left\{s|x-y|+t|y|:x-y\in \lambda B^n,y\in\mu B^n\right\}\\
     &=\begin{cases}s|x|,& |x|\leq \lambda,\\
	s\lambda+t(|x|-\lambda),& \lambda <|x|\leq \lambda+\mu,\\
	+\infty,& \lambda+\mu<|x|,\end{cases}.
 \end{align*}
 Hence,
	$$\proj_E\left((\lambda\sq u_s)\infconv \big((\mu \sq u_t)\circ \vartheta^{-1}\big)\right) (x_E) = \begin{cases}s|x_E|,& |x_E|\leq \lambda,\\
	s\lambda+t(|x_E|-\lambda),& \lambda <|x_E|\leq \lambda+\mu,\\
	+\infty,& \lambda+\mu<|x_E|,\end{cases}$$
	for every $\vartheta\in \SO(n)$ and $E\in\Grass{j}{n}$. Thus, it follows from \eqref{eq:def_oz_u_v} and Lemma~\ref{le:ck_map} (for $j=n$ the lemma holds trivially) that
	\begin{align}\label{eq:Zusut}
	\oZ(\lambda\sq u_s,\mu \sq u_t)&=\int_{\Rn}\alpha(|y|)\d\MAp_j\big((\lambda\sq u_s)\infconv (\mu \sq u_t);y\big)\nonumber\\
	&= \frac{\kappa_n}{\kappa_j}\int_{\Grass{j}{n}}\int_E \alpha(|y_E|)\d\MAp_E \big(\proj_E \big((\lambda\sq u_s)\infconv (\mu \sq u_t)\big);y_E\big) \d E\nonumber\\
	&=\frac{\kappa_n}{\kappa_j} \int_{\Grass{j}{n}}\int_{(\lambda+\mu)B^n} \alpha\left(\left\vert\nabla\proj_E\big((\lambda\sq u_s)\infconv (\mu \sq u_t)\big)(x_E) \right|\right) \d x_E \d E\nonumber\\
	&=\kappa_n \left(\lambda^j \alpha(s) + \big((\lambda+\mu)^j-\lambda^j\big) \alpha(t)\right)\nonumber\\
	&= \kappa_n \left(\lambda^j \alpha(s)+\sum_{i=0}^{j-1}\binom{j}{i}\lambda^i\mu^{j-i} \alpha(t) \right)
	\end{align}
    for $0\leq s\leq t$.
	On the other hand, by \eqref{eq:oz_alpha_ik}, \eqref{eq:map_bn}, and Lemma~\ref{le:int_map_ut},
	\begin{multline}\label{eq:Zusut2}
	\oZ(\lambda\sq u_s,\mu \sq u_t)\\
    = \kappa_n^2\left(\alpha_{0,0}(0)+\sum_{k=1}^n \mu^k \alpha_{0,k}(t) \right) + \kappa_n^2\sum_{i=1}^n\lambda^i \left(\alpha_{i,0}(s,0)+\sum_{k=1}^n \mu^k \alpha_{i,k}(s,t)\right).
	\end{multline}
	Since $\lambda,\mu\geq 0$ were arbitrary, we can compare coefficients of the last two equations to obtain
	$$
	\kappa_n \alpha_{j,0}(s,0)=\alpha(s),\quad \kappa_n \alpha_{0,j}(t)=\alpha(t),\quad \kappa_n \alpha_{i,j-i}(s,t)=\binom{j}{i}\alpha(t),
	$$
    for $1\leq i \leq j-1$ and $\alpha_{i,k}(s,t)=0$ if $i+k\neq j$ whenever $0\leq s \leq t$.
	
	Proceeding for the case $s\ge t$  similarly as in the derivation of \eqref{eq:Zusut}, we  obtain
$$
\oZ(\lambda\sq u_s,\mu \sq u_t) =\kappa_n \left(\mu^j \alpha(t)+\sum_{i=0}^{j-1}\binom{j}{i}\mu^i\lambda^{j-i} \alpha(s) \right).
$$
A comparison with the coefficients available from \eqref{eq:Zusut2} shows that 
	$$\kappa_n \alpha_{j,0}(s,0)=\alpha(s), \quad \kappa_n \alpha_{0,j}(t)=\alpha(t),\quad \kappa_n \alpha_{i,j-i}(s,t)=\binom{j}{i}\alpha(s),$$
    for $1\leq i \leq j-1$ and $\alpha_{i,k}(s,t)=0$ if $i+k\neq j$, also for $s\ge t$. 
    
    The claim now follows after considering \eqref{eq:map_bn}.
\end{proof}

\section{Further Formulas}
\label{se:further_results}
For $\alpha\in C_c({[0,\infty)})$ and $0\leq j\leq n$, we define the $j$th \emph{functional intrinsic volume} on $\fconvs$ with density $\alpha$, denoted by $\oZZb{j}{\alpha}$, as
\begin{equation}
\label{eq:ozzb_map}
\oZZb{j}{\alpha}(u)=\binom{n}{j}\frac{1}{\kappa_{n-j}} \int_{\Rn} \alpha(|y|)\d\MAp_j(u;y)
\end{equation}
for $u\in\fconvs$. Clearly, $\oZZb{j}{\alpha}^*(v)=\oZZb{j}{\alpha}(v^*)$ for every $v\in\fconvf$.

By \eqref{eq:ozzb_map}, we have the following reformulation of Theorem~\ref{thm:add_kin_form_map}.

\begin{theorem}
\label{thm:kin_bis}
If $0\leq j\leq n$ and $\alpha\in C_c([0,\infty))$, then
\begin{multline}
\label{eq:kin_bis}
\kappa_n \kappa_{n-j} \int_{\SO(n)} \oZZb{j}{\alpha}\big(u\infconv (v\circ\vartheta^{-1})\big) \d\vartheta\\
=\binom{n}{j} \sum_{i=0}^j \binom{j}{i} \int_{\Rn}\int_{\Rn} \alpha(\max\{|x|,|y|\})\d\MAp_{j-i}(v;y)\d\MAp_i(u;x)
\end{multline}
for $u,v\in\fconvs$.
\end{theorem}

If in \eqref{eq:kin_bis} we choose $v$ to be the convex indicator function $\ind_L$ of a convex body $L \in \Kn$, then a direct application of Lemma \ref{le:MAp_bodies} gives a specialization of Theorem~\ref{thm:kin_bis} reading as follows, where we write
$$
\stirling{k}{j}:=\binom{k}{j}\frac{\kappa_k}{\kappa_j \kappa_{k-j}}
$$
for $j,k\in\N$.

\begin{corollary}
\label{cor:add_kin_ind}
    If $0\leq j\leq n$ and $\alpha\in C_c([0,\infty))$, then
    \begin{equation*}
        \int_{\SO(n)}\oZZb{j}\alpha (u \infconv \ind_{\vartheta L}) \d \vartheta =\sum_{i=0}^j \stirling{n-i}{j-i}{\stirling{n}{j-i}}^{-1}\oZZb{i}\alpha (u) V_{j-i}(L)
    \end{equation*}
    for $u \in \fconvs$ and $L \in \K^n$. 
\end{corollary}
\begin{proof}
    It follows from Lemma~\ref{le:MAp_bodies} that for $u \in \fconvs$, $L \in \K^n$, and $0\leq j \leq n$, we have \[\binom{n}{j}\MAp_j(\ind_L;B)=\kappa_{n-j}V_j(L)\delta_0(B)\] for Borel sets $B \subseteq \Rn$. Applying Theorem~\ref{thm:kin_bis}, we then infer 
    \begin{align*}
        \int_{\SO(n)}&\oZZb{j}{\alpha}(u\infconv I_{\vartheta L})\d \vartheta\\
        &=\frac{\binom{n}{j}}{\kappa_n \kappa_{n-j}} \sum_{i=0}^j \binom{j}{i} \int_{\Rn} \int_{\Rn} \alpha(\max \{|x|,|y|\})\d \MAp_{j-i}(\ind_L;y)\d \MAp_i(u;x)\\
        &=\frac{\binom{n}{j}}{\kappa_n \kappa_{n-j}}\sum_{i=0}^j \frac{\binom{j}{i}\kappa_{n-(j-i)}}{\binom{n}{j-i}} V_{j-i}(L)\int_{\Rn}\alpha(|x|)\d \MAp_i(u;x)\\
        &=\sum_{i=0}^j\frac{\binom{n}{j}\binom{j}{i}}{\binom{n}{j-i}\binom{n}{i}}\frac{\kappa_{n-i}\kappa_{n-(j-i)}}{\kappa_n \kappa_{n-j}}\oZZb{i}{\alpha}(u)V_{j-i}(L)\\
        &=\sum_{i=0}^j \stirling{n-i}{j-i}{\stirling{n}{j-i}}^{-1}\oZZb{i}\alpha (u) V_{j-i}(L),
    \end{align*}
    where we used that $\alpha(\max\{|x|,0\})=\alpha(|x|)$ for $x\in\Rn$ and
    $$
    \frac{\binom{n}{j}\binom{j}{i}}{\binom{n}{i}}=\binom{n-i}{j-i}
    $$
    for $0\leq i\leq j\leq n$.
\end{proof}

Note that if in the last result we furthermore also choose $u$ to be the indicator function of a convex body $K\in\Kn$, then we recover the additive kinematic formula for convex bodies \eqref{eq:additive_kinematic_formula}, which can be written as
$$
\int_{\SO(n)} V_j(K+\vartheta L)\d\vartheta=\sum_{i=0}^j \stirling{n-i}{j-i}{\stirling{n}{j-i}}^{-1}V_i(K) V_{j-i}(L).
$$

Next, we consider mixed functionals. A corollary of \eqref{eq:additive_kinematic_formula} for the mixed volume $V:(\Kn)^n\to\R$ states that if $0\leq j\leq n$, then
\begin{equation}
\label{eq:formula_mixed_volume}
\binom{n}{j} \int_{\SO(n)} V(K[j],\vartheta L [n-j]) \d \vartheta = \stirling{n}{j}^{-1} V_j(K) V_{n-j}(L)
\end{equation}
for every $K,L\in\Kn$. See, for example, formula (6.7) in \cite{schneider_weil}.

For $\alpha\in C_c({[0,\infty)})$ we define the operator $\oVb_\alpha$ on $(\fconvs)^n$ as
$$
\oVb_\alpha(u_1,\ldots,u_n)= \int_{\Rn} \alpha(|y|)\d \MAp(u_1,\ldots,u_n;y)
$$
for $u_1,\ldots,u_n\in\fconvs$. Clearly, by the properties of the conjugate mixed Monge--Amp\`ere measure, the functional $\oVb_\alpha$ is symmetric in its entries. Moreover, in each of its arguments it is continuous with respect to epi-convergence, epi-homogeneous of degree 1, and epi-translation invariant. We remark that functionals of this form were also treated in \cite{Alesker_19} and, from a valuation point of view, in \cite{Colesanti-Ludwig-Mussnig-4,Knoerr_support}.

By Lemma~\ref{le:MAp_bodies}, it is immediate to check that $\oVb_\alpha$ generalizes the classical mixed volumes, i.e.,
$$\oVb_\alpha(\ind_{K_1},\ldots,\ind_{K_n})=\alpha(0)V(K_1,\dots,K_n)
$$
for $K_1,\dots,K_n \in \Kn$. We have the following functional integral formula which includes \eqref{eq:formula_mixed_volume} in the special case where $k=n$, $u=\ind_K$, and $v=\ind_L$ (see also \cite[Lemma 5.8]{Hug_Weil_Lectures}).

\begin{corollary}
\label{cor:func_mix}
If $0\leq j\leq k\le  n$ and $\alpha\in C_c({[0,\infty)})$, then
\begin{align*}
&\int_{\SO(n)} \oVb_{\alpha}\big(u_1,\ldots,u_j,v_1\circ \vartheta^{-1},\ldots,v_{k-j}\circ \vartheta^{-1},\ind_{B^n}[n-k]\big)\d\vartheta\\
&\;= \frac{1}{\kappa_n} \int_{\Rn}\int_{\Rn} \alpha(\max\{|x|,|y|\}) \d\MAp_{j}(u_1,\ldots,u_j;x) \d \MAp_{k-j}(v_1,\ldots,v_{k-j};y)
\end{align*}
for $u_1,\ldots,u_j,v_1,\ldots,v_{k-j}\in\fconvs$.
\end{corollary}
\begin{proof}
Let $u,v\in\fconvs$. For $\varepsilon>0$ it follows from Theorem~\ref{thm:kin_bis} and the properties of the measures $\MAp_{n-j}(v;\cdot)$ that
\begin{multline*}
\int_{\SO(n)} \oZZb{k}{\alpha}\big(u\infconv \varepsilon\sq (v\circ\vartheta^{-1})\big) \d\vartheta\\
= \frac{1}{\kappa_n\kappa_{n-k}} \binom{n}{k}\sum_{i=0}^k \binom{k}{i} \varepsilon^{k-i} \int_{\Rn} \int_{\Rn} \alpha(\max\{|x|,|y|\}) \d\MAp_{k-i}(v;y)\d\MAp_{i}(u;x)
\end{multline*}
for every $u,v\in\fconvs$. On the other hand, by \eqref{eq:cma_poly_exp} we have
\begin{align*}
&\int_{\SO(n)} \oZZb{k}{\alpha}\big(u\infconv \varepsilon\sq (v\circ\vartheta^{-1})\big) \d\vartheta\\
&= \frac{1}{\kappa_{n-k}} \binom{n}{k} \int_{\SO(n)} \int_{\Rn} \alpha(|x|) \d\MA^*_k\big(u\infconv \varepsilon\sq (v\circ\vartheta^{-1});x\big)\d\vartheta\\
&= \frac{1}{\kappa_{n-k}} \binom{n}{k} \sum_{i=0}^k \binom{k}{i}\varepsilon^{k-i} \int_{\SO(n)}\int_{\Rn} \alpha(|x|) \d\MA^*\big(u[i],v\circ \vartheta^{-1}[k-i],\ind_{B^n}[n-k];x\big)\d\vartheta\\
&= \frac{1}{\kappa_{n-k}} \binom{n}{k} \sum_{i=0}^k \binom{k}{i}\varepsilon^{k-i} \int_{\SO(n)} \oVb_{\alpha}\big(u[i],v\circ\vartheta^{-1}[k-i],\ind_{B^n}[n-k]\big)\d\vartheta.
\end{align*}
The result now follows after comparing coefficients together with multilinearity.
\end{proof}

As an application, Corollary \ref{cor:func_mix} can be used to obtain generalizations of further formulas. In particular, mimicking the so-called \emph{Minkowski difference} (see, for example, \cite[Note 3 of Section 6.1]{schneider_weil}), we can introduce the operation of inf-deconvolution. If for $u,v \in \fconvs$, there exists $w \in \fconvs$ such that
$$
w \infconv v = u,
$$
then we say that $w$ is the \emph{inf-deconvolution} of $u$ and $v$, which we denote by
$$
w=u \diffconv v.
$$
Equivalently, this means that $u\diffconv v$ exists if and only if the (pointwise) difference $u^*-v^*$ is an element of $\fconvf$ and
$$(u\diffconv v)^*=u^*-v^*.$$
Moreover, we say that $v$ \textit{rolls freely} in $u$ if for every $\vartheta \in \SO(n)$ the expression $u \diffconv (v \circ \vartheta^{-1})$ is well-defined. With this new terminology at hand, we obtain the following consequence of Corollary \ref{cor:func_mix}.

\begin{corollary}
\label{cor:diffconv}
Let $0\le k \le n$, $\alpha\in C_c({[0,\infty)})$, and $u,v \in \fconvs$. If $v$ rolls freely in $u$, then 
\begin{align}
\label{eqdiffconv}
   & \int_{\SO(n)} \oZZb{k}{\alpha}\big(u \diffconv (v\circ \vartheta^{-1})\big) d\vartheta\\
   &=\frac{\binom{n}{k}}{\kappa_n\kappa_{n-k}} \sum_{j=0}^k(-1)^{k-j} \binom{k}{j} \int_{\R^n} \int_{\R^n} \alpha(\max\{|x|,|y|\})\d \MAp_{k-j}(v;y)\d\MAp_{j}(u;x).\nonumber
\end{align}
\end{corollary}
\begin{proof}
Since $\oVb_\alpha$ is linear with respect to inf-convolution in each of its arguments and presuming that $u\diffconv (v\circ \vartheta^{-1})$ exists for every $\vartheta\in\SO(n)$, we have
\begin{align*}
&\oZZb{k}{\alpha}\big(u \diffconv (v\circ \vartheta^{-1})\big)\\
&=\oVb_\alpha\big(u \diffconv (v\circ \vartheta^{-1})[k],\ind_{B^{n}}[n-k]\big)\\
&=\oVb_\alpha\big(u \diffconv (v\circ \vartheta^{-1})[k-1],u,\ind_{B^{n}}[n-k]\big)\\
&\qquad -\oVb_{\alpha}\big(u \diffconv (v\circ \vartheta^{-1})[n-1],(v\circ \vartheta^{-1}),\ind_{B^{n}}[n-k]\big)\\
&\;\;\vdots\\
&=\sum_{j=0}^k (-1)^{k-j}\binom{k}{j}\oVb_\alpha\big(u[j],(v\circ \vartheta^{-1})[k-j],\ind_{B^{n}}[n-k]\big),
\end{align*}
which can be proved in detail by induction on $k\in\{0,\ldots,n\}$. 
Integration over $\SO(n)$ together with an application of Corollary \ref{cor:func_mix} results in relation \eqref{eqdiffconv}.
\end{proof}

\section{Singular Hessian Integrals}
\label{se:singular_hessian}
In this section, we demonstrate another application of the special case $k=n$ of  Corollary~\ref{cor:func_mix}. Let us first state its equivalent version on $\fconvf$ (for $k=n$), where we write $\MA_j(v;\cdot)=\MA(v[j],h_{B^n}[n-j];\cdot)$ for $v\in\fconvf$ and $0\leq j\leq n$.

\begin{corollary}
\label{cor:func_mix_dual}
If $0\leq j\leq n$ and $\alpha\in C_c({[0,\infty)})$, then
\begin{multline*}
\int_{\SO(n)} \int_{\Rn} \alpha(|x|)\d\MA\big(v[j],w\circ\vartheta^{-1}[n-j];x\big) \d\vartheta\\
=\frac{1}{\kappa_n} \int_{\Rn}\int_{\Rn} \alpha(\max\{|x|,|y|\})\d\MA_{n-j}(w;y)\d\MA_j(v;x)
\end{multline*}
for $v,w\in\fconvf$.
\end{corollary}

As mentioned in Section~\ref{se:introduction}, functional intrinsic volumes were previously defined in terms of Hessian measures. For this, let
$$D_j^n=\left\{\zeta\in C_b((0,\infty)) : \lim_{s\to 0^+} s^{n-j}\zeta(s)=0, \exists \lim_{s\to 0^+} \int_s^\infty t^{n-j-1}\zeta(t)\d t \in\R \right\}$$
for $0\leq j\leq n-1$, where $C_b((0,\infty))$ denotes the set of continuous function with bounded support on $(0,\infty)$. In addition, set 
$$D_n^n=\left\{\zeta\in C_b((0,\infty)): \exists \lim_{s\to 0^+} \zeta(s) \in \R\right\},$$
which we identify with $C_c({[0,\infty)})$. In \cite[Theorem 1.4]{Colesanti-Ludwig-Mussnig-5} and later also in \cite{Colesanti-Ludwig-Mussnig-6,Colesanti-Ludwig-Mussnig-7,Knoerr_singular}, it was shown that for $0\leq j\leq n$ and $\zeta\in D_j^n$, the map
\begin{equation}
\label{eq:hess_extend}
v\mapsto \int_{\Rn} \zeta(|x|) [\Hess v(x)]_j \d x
\end{equation}
continuously extends from $\fconvf\cap C_+^2(\Rn)$ to $\fconvf$. This extension was used as the original definition for functional intrinsic volumes, meaning they can be understood as singular Hessian integrals.

In \cite{Colesanti-Ludwig-Mussnig-7}, an alternative proof of existence for the continuous extension of \eqref{eq:hess_extend} was found. The essential observation (see \cite[Proposition 6.7]{Colesanti-Ludwig-Mussnig-7}) is that
\begin{equation}
\label{eq:connection_hess_ma_cR}
\int_{\Rn} \zeta(|x|)\d\MA(v[j],q[n-j];x) =  \int_{\Rn} \cR^{n-j} \zeta(|x|)\d\MA_j(v;x)
\end{equation}
for $0\leq j\leq n-1$, $\zeta\in D_j^n$, and $v\in C^2(\Rn)$, where
$$
\cR^{n-j} \zeta(s) = s^{n-j} \zeta(s) + (n-j) \int_s^\infty t^{n-j-1}\zeta(t) \d t
$$
for $s>0$ and where $q(x)=|x|^2/2$. If we consistently define $\cR^0\zeta=\zeta$ for $\zeta\in D^n_n$, then \eqref{eq:connection_hess_ma_cR} remains true also for $j=n$. In addition, it was previously shown in \cite[Lemma 3.8]{Colesanti-Ludwig-Mussnig-6}, that $\cR^{n-j}$ is a bijection from $D_j^n$ to $D_n^n$. Together with 
\begin{equation}
\label{eq:connection_MA_Hessian}
\binom{n}{j}\d\MA(v[j],q[n-j];x)=[\Hess v(x)]_j \d x
\end{equation}
for $v\in \fconvf\cap C^2(\Rn)$ and $0\leq j\leq n$, this then implies that \eqref{eq:hess_extend} continuously extends to $\fconvf$ (for $j=0$ we use the convention $[\Hess v(x)]_0\equiv 1$).

As we illustrate in the following, Corollary \ref{cor:func_mix_dual} gives a straightforward understanding of \eqref{eq:connection_hess_ma_cR}. Indeed, if in Corollary \ref{cor:func_mix_dual} we choose $w=q$, then it follows from the rotational symmetry of $q$ that
\begin{equation}
\label{eq:connection_hess_ma_cR2}
\int_{\Rn} \zeta(|x|) \d\MA(v[j],q[n-j];x) = \int_{\Rn} \beta(|x|) \d\MA_j(v;x)
\end{equation}
for every $0\leq j\leq n$, $\zeta\in D_n^n$, and $v\in\fconvf$,
where
$$\beta(|x|)=\frac{1}{\kappa_n} \int_{\Rn} \zeta(\max\{|x|,|y|\}) \d\MA_{n-j}(q;y)$$
for $x\in\Rn$. If $j=n$, then $\MA_0(q;\cdot)=\kappa_n\delta_0$ (see \cite[Lemma 4.6]{Hug_Mussnig_Ulivelli_supports}) implies that $\beta(|x|)=\zeta(|x|)$. If $0\le j\le n-1$, then direct calculations (see also \cite[Theorem 4.5 (d)]{Colesanti-Ludwig-Mussnig-7}) show that 
$$\d\MA_{n-j}(q;y)=\frac{n-j}{n}\frac{1}{|y|^j} \d y$$
on $\Rn\backslash\{o\}$ and, by \cite[Lemma 4.3]{Hug_Mussnig_Ulivelli_supports}, $\MA_{n-j}(q;\{o\})=V(\{o\}[n-j],B^n[j])=0$. Therefore,
\begin{align*}
\beta(|x|)&=\frac{1}{\kappa_n} \left(\zeta(|x|) \int_{|x|B^n} \d\MA_{n-j}(q;y) + \int_{\Rn\backslash |x|B^n} \zeta(|y|)\d\MA_{n-j}(q;y) \right)\\
&= \frac{1}{\kappa_n} \left( \zeta(|x|) \frac{n-j}{n} (n \kappa_n) \int_0^{|x|} r^{n-1-j} \d r + \frac{n-j}{n} (n \kappa_n) \int_{|x|}^{\infty} \zeta(r) r^{n-1-j} \d r\right)\\
&= \zeta(|x|)|x|^{n-j} + (n-j) \int_{|x|}^{\infty} \zeta(r) r^{n-j-1}\d r\\
&= \cR^{n-j}\zeta(|x|)
\end{align*}
for $x\in\Rn$. By using the properties of the transform $\cR^{n-j}$ and of the measures ${\MA(v[j],q[n-j];\cdot)}$ for $v\in C^2(\Rn)$ (see, for example, \cite[Lemma 3.1]{Colesanti-Ludwig-Mussnig-5}), one can now extend \eqref{eq:connection_hess_ma_cR2} from $\zeta\in D_n^n$ to $\zeta\in D_j^n$.

We remark that using the same method as above, similar relations can be obtained between integrals with respect to $\MA_j(v;\cdot)$ and integrals with respect to $\MA(v[j],w[n-j];\cdot)$, where $w \in \fconvf$ is rotationally symmetric.

\section{Formulas for Convex Bodies}
\label{se:formulas_bodies}
For the proof of Theorem~\ref{thm:formula_mx_area_meas} we use a connection between conjugate mixed Monge--Amp\`ere measures and mixed area measures which was established in \cite[Section 3]{Knoerr_Ulivelli} and further expanded upon in \cite[Section 4.2]{Hug_Mussnig_Ulivelli_supports}.

Let $n\geq 2$ and let $\proj_H\colon \Rn\to H$ denote the orthogonal projection onto the hyperplane $H=e_n^{\perp}$, which we will identify with $\R^{n-1}$. To each convex body $K\in\Kn$ we assign the function
$$\lfloor K \rfloor(x) = \begin{cases}
\min\{t\in\R : (x,t)\in K\}\quad&\text{if } x\in \proj_H K,\\ \infty\quad&\text{else.} \end{cases}
$$
This defines a lower semicontinuous, convex function on $\R^{n-1}$ with compact domain, and in particular, $\lfloor K \rfloor$ is an element of $\fconvsH$. Furthermore, observe that
\begin{equation}
\label{eq:floor_mink_sum_infconv}
\lfloor K + L \rfloor = \lfloor K \rfloor \infconv \lfloor L \rfloor
\end{equation}
for $K,L\in\Kn$.

In the following, we denote by
$\s^{n-1}_{-}=\{z\in\sn \colon  \langle z,e_{n}\rangle <0\}$ the lower half-sphere in $\Rn$. The gnomonic projection $\gnom\colon \s^{n-1}_-\to \R^{n-1}$ is defined by
$$\gnom(z)=\frac{(z_1,\ldots,z_{n-1})}{|z_{n}|}$$
for $z=(z_1,\ldots,z_n)\in \s^{n-1}_-$, with inverse
$$\gnom^{-1}(x)=\frac{(x,-1)}{\sqrt{1+|x|^2}}$$
for $x\in\R^{n-1}$. By \cite[Section 3.2]{Knoerr_Ulivelli} and \cite[Remark 4.7]{Hug_Mussnig_Ulivelli_supports} we have the following reformulation of \cite[Corollary 4.9]{Hug_Mussnig_Ulivelli_supports}.

\begin{lemma}
\label{le:mixed_ma_s}
If $\varphi\colon \s^{n-1}_{-} \to [0,\infty)$ is measurable, then
$$\int_{\R^{n-1}} \varphi\left(\gnom^{-1}(y)\right) \d\MAp(\lfloor K_1\rfloor,\ldots,\lfloor K_{n-1}\rfloor ;y)=\int_{\s_-^n} |\langle z,e_{n}\rangle| \varphi(z) \d S(K_1,\ldots,K_{n-1},z)$$
for $K_1,\ldots,K_{n-1}\in \Kn$.
\end{lemma}

\begin{proof}[\textbf{Proof of Theorem~\ref{thm:formula_mx_area_meas}}]
Let $\beta \colon [0,1]\to [0,\infty)$ be measurable and let $\tilde\beta\colon [0,1]\to[0,\infty)$ be given by $\tilde\beta (t)=t\beta(t)$ for $t\in [0,1]$. Since $\tilde\beta(|z_n|)=0$ if $z_n=\langle z,e_n\rangle=0$, we thus obtain
\begin{align}
   &\int_{\sn} |z_n| \beta(|z_n|) \d S((K+\vartheta L)  [j],B_H^{n-1}[n-1-j],z) \notag
   \\
   &\quad = \int_{\mathbb{S}^{n-1}_-}   \tilde\beta(|z_n|) \d S((K+\vartheta L) [j],B_H^{n-1}[n-1-j],z)\label{eq:line_2}\\ 
   & \quad\quad+ \int_{\mathbb{S}^{n-1}_+}  \tilde\beta(|z_n|) \d S((K+\vartheta L)[j],B_H^{n-1}[n-1-j],z), \label{eq:line_3}
\end{align}
for $\vartheta\in \SO(n-1)$, where $\s^{n-1}_{+}=\{z\in\s^{n-1} \colon  \langle z,e_{n}\rangle >0\}$. 

Next, we want to obtain suitable representations for the integrals in \eqref{eq:line_2} and \eqref{eq:line_3} so that we can apply Theorem~\ref{thm:add_kin_form_map}. Notice that the integral in \eqref{eq:line_3} can be rewritten as an integral on $\s^{n-1}_-$. Indeed, if we denote by $\bar{K}$ and $\bar{L}$ the reflections of $K$ and $L$ through $H$, respectively, then \eqref{eq:line_3} can be written as
$$
\int_{\s^{n-1}_-} \tilde\beta(|z_n|) \d S((\bar{K}+\vartheta \bar{L}) [j],B_H^{n-1}[n-1-j],z).
$$
Here we used that the considered reflection fixes $B^{n-1}_H$ and elements of $\SO(n-1)$.

Let $u,v \in \mathrm{Conv}_{\mathrm{cd}}(\R^{n-1})$ be given by $u=\lfloor K \rfloor$ and $v=\lfloor L \rfloor$. Furthermore, let the measurable function $\alpha\colon [0,\infty)\to [0,\infty)$ be defined by the relation
\begin{align*}
\alpha(t)=\beta\left(\frac{1}{\sqrt{1+t^2}}\right)\quad \text{or equivalently}\quad
\beta(s)=\alpha\left( \frac{\sqrt{1-s^2}}{s} \right)
\end{align*}
for $t\in [0,\infty)$ and $s\in (0,1]$. By \eqref{eq:floor_mink_sum_infconv} and Lemma~\ref{le:mixed_ma_s}, applied with $\varphi(z)=\beta(|\langle z,e_n\rangle|)$ for $z\in \s^{n-1}_{-}$, we now have
\begin{align*}
&\int_{\s^{n-1}_-} \tilde\beta(|z_n|) \d S((K+\vartheta L) [j],B_H^{n-1}[n-1-j],z)\\
&\quad=\int_{\R^{n-1}} \alpha(|y|)\d\MAp_j(u \infconv (v \circ \vartheta^{-1});y),
\end{align*}
where we have used that $v \circ \vartheta^{-1}=\lfloor \vartheta L \rfloor$ and $\lfloor B_H^{n-1}\rfloor = \ind_{B^{n-1}}$. For $\bar{u}=\lfloor \bar{K} \rfloor, \bar{v}=\lfloor \bar{L} \rfloor$, we obtain analogously
\begin{align*}
    &\int_{\s^{n-1}_+}  \tilde\beta(|z_n|) \d S((K+\vartheta L) [j],B_H^{n-1}[n-1-j],z)\\
    &\quad=\int_{\s^{n-1}_-}  \tilde\beta (|z_n|) \d S((\bar{K}+\vartheta \bar{L}) [j],B_H^{n-1}[n-1-j],z)\\
    &\quad=\int_{\R^{n-1}} \alpha(|y|)\d\MAp_j(\bar{u} \infconv (\bar{v} \circ \vartheta^{-1});y).
\end{align*}
Hence we get
\begin{align}
    \label{eq:kin_for_body}
    &\int_{\sn} \tilde\beta(|z_n|) \d S((K+\vartheta L) [j],B_H^{n-1}[n-1-j],z)\\
    &\quad=\int_{\R^{n-1}} \alpha(|y|)\d\MAp_j(u \infconv (v \circ \vartheta^{-1});y) + \int_{\R^{n-1}} \alpha(|y|)\d\MAp_j(\bar{u} \infconv (\bar{v} \circ \vartheta^{-1});y).\notag
\end{align}
We now integrate \eqref{eq:kin_for_body} over $\SO(n-1)$ with respect to the Haar probability measure. Together with Theorem~\ref{thm:add_kin_form_map}, applied with respect to the ambient space $\R^{n-1}$, we infer
\begin{align}
   & \int_{\SO(n-1)} \int_{\s^{n-1}} \tilde\beta(|z_n|) \d S((K+\vartheta L) [j],B_H^{n-1}[n-1-j],z) \d\vartheta\notag \\
    &\quad= \int_{\SO(n-1)}\int_{\R^{n-1}} \alpha(|y|)\d\MAp_j(u \infconv (v \circ \vartheta^{-1});y) \d\vartheta\notag\\
    &\quad\quad +\int_{\SO(n-1)}\int_{\R^{n-1}} \alpha(|y|)\d\MAp_j(\bar{u} \infconv (\bar{v} \circ \vartheta^{-1});y) \d\vartheta\notag\\
    &\quad= \frac{1}{\kappa_{n-1}} \sum_{i=0}^j \binom{j}{i} \int_{\R^{n-1}}\int_{\R^{n-1}} \alpha(\max\{|x|,|y|\})\d\MAp_{j-i}(v;y)\d\MAp_i(u;x)\notag\\
    &\quad\quad+ \frac{1}{\kappa_{n-1}} \sum_{i=0}^j \binom{j}{i} \int_{\R^{n-1}}\int_{\R^{n-1}} \alpha(\max\{|x|,|y|\})\d\MAp_{j-i}(\bar{v};y)\d\MAp_i(\bar{u};x)\notag\\
\begin{split}
\label{eq:kin_body_A}
    &\quad= \frac{1}{\kappa_{n-1}} \sum_{i=0}^j \binom{j}{i} \int_{\s^{n-1}_-}\int_{\s^{n-1}_-} |w_n| |z_n| \beta(\min\{|w_n|,|z_n|\})\\
    &\quad\qquad\quad \d S(L [j-i],B_H^{n-1}[n-1-j+i],w) \d S(K [i],B_H^{n-1}[n-1-i],z)\\
    &\quad\quad +\frac{1}{\kappa_{n-1}} \sum_{i=0}^j \binom{j}{i} \int_{\s^{n-1}_-}\int_{\s^{n-1}_-} |w_n| |z_n| \beta(\min\{|w_n|,|z_n|\})\\
    &\quad\qquad\quad \d S(\bar{L} [j-i],B_H^{n-1}[n-1-j+i],w) \d S(\bar{K} [i],B_H^{n-1}[n-1-i],z).
\end{split}
\end{align}
In the last step we used Lemma~\ref{le:mixed_ma_s} together with the fact that $s\mapsto \sqrt{1-s^2}/s$ is decreasing and thus
$$\beta(\min\{a,b\})=\alpha\left(\max\left\{\frac{\sqrt{1-a^2}}{a},\frac{\sqrt{1-b^2}}{b}\right\} \right)$$
for $a,b\in (0,1]$. Observe that the last integral in \eqref{eq:kin_body_A} can be rewritten as
\begin{align}
\begin{split}
\label{eq:kin_body_B}
\int_{\s^{n-1}_-}&\int_{\s^{n-1}_-} |w_n| |z_n| \beta(\min\{|w_n|,|z_n|\})\\
&\qquad\quad \d S(\bar{L} [j-i],B_H^{n-1}[n-1-j+i],w) \d S(\bar{K} [i],B_H^{n-1}[n-1-i],z)\\
&=\int_{\s^{n-1}_+}\int_{\s^{n-1}_+} |w_n| |z_n| \beta(\min\{|w_n|,|z_n|\})\\
&\qquad\quad \d S(L [j-i],B_H^{n-1}[n-1-j+i],w) \d S(K [i],B_H^{n-1}[n-1-i],z)
\end{split}
\end{align}
for $0\leq i\leq j$. Similar to the above, we obtain
\begin{align}
\begin{split}
\label{eq:kin_body_C}
    &\int_{\SO(n-1)} \int_{\s^{n-1}}\tilde\beta(|z_n|) \d S((K+\vartheta \bar{L}) [j],B_H^{n-1}[n-1-j],z) \d\vartheta\\
    &\quad= \frac{1}{\kappa_{n-1}} \sum_{i=0}^j \binom{j}{i} \int_{\s^{n-1}_+}\int_{\s^{n-1}_-} |w_n| |z_n| \beta(\min\{|w_n|,|z_n|\})\\
    &\quad\qquad\quad \d S(L [j-i],B_H^{n-1}[n-1-j+i],w) \d S(K [i],B_H^{n-1}[n-1-i],z)\\
    &\quad\quad +\frac{1}{\kappa_{n-1}} \sum_{i=0}^j \binom{j}{i} \int_{\s^{n-1}_-}\int_{\s^{n-1}_+} |w_n| |z_n| \beta(\min\{|w_n|,|z_n|\})\\
    &\quad\qquad\quad \d S(L [j-i],B_H^{n-1}[n-1-j+i],w) \d S(K [i],B_H^{n-1}[n-1-i],z).
\end{split}
\end{align}
Thus, combining \eqref{eq:kin_body_A}, \eqref{eq:kin_body_B}, and \eqref{eq:kin_body_C}, we obtain
\begin{align}
    \begin{split}
        &\int_{\SO(n-1)\times \oO(1)} \int_{\s^{n-1}} \tilde\beta(|z_n|) \d S((K+\eta L) [j],B_H^{n-1}[n-1-j],z) \d\eta\\
&\quad= \frac 12 \int_{\SO(n-1)} \int_{\s^{n-1}} \tilde\beta(|z_n|) \d S((K+\vartheta L) [j],B_H^{n-1}[n-1-j],z) \d\vartheta\\
&\quad\quad+ \frac 12 \int_{\SO(n-1)} \int_{\s^{n-1}} \tilde\beta(|z_n|) \d S((K+\vartheta \bar{L}) [j],B_H^{n-1}[n-1-j],z) \d\vartheta\\
&\quad= \frac{1}{2\kappa_{n-1}} \sum_{i=0}^j \binom{j}{i} \int_{\sn}\int_{\sn} |w_n| |z_n| \beta(\min\{|w_n|,|z_n|\})\\
&\quad\qquad\quad \d S(L [j-i],B_H^{n-1}[n-1-j+i],w) \d S(K [i],B_H^{n-1}[n-1-i],z), \label{eq:kin_body_D}
    \end{split}
\end{align}
where we used that $|w_n| |z_n| \beta(\min\{|w_n|,|z_n|\})=0$  if $z\in\sn \cap e_n^{\perp}$ or $w\in\sn \cap e_n^{\perp}$. This concludes the proof.
\end{proof}

Arguing similarly as in the proof of Corollary \ref{cor:func_mix}, we obtain from Theorem~\ref{thm:formula_mx_area_meas} the following equivalent version.

\begin{corollary}
\label{corformula_mx_area_meas}
Let $n\geq 2$. If $0\leq i\leq j\leq n-1$ and $\beta \colon [0,1]\to [0,\infty)$ is measurable, then
\begin{align*}
    &\int_{\SO(n-1)\times \oO(1)} \int_{\sn} |z_n| \beta(|z_n|) \d S(K_1,\ldots,K_i,\eta L_1,\ldots,\eta L_{j-i}),B_H^{n-1}[n-1-j],z) \d\eta\\
&\quad= \frac{1}{2\kappa_{n-1}}   \int_{\sn}\int_{\sn} |w_n| |z_n| \beta(\min\{|w_n|,|z_n|\}) \d S(K_1,\ldots,K_i,B_H^{n-1}[n-1-i],z)\\
&\qquad\qquad \d S(L_1,\ldots,L_{j-i},B_H^{n-1}[n-1-j+i],w) 
\end{align*}
for $K_1,\ldots,K_i,L_1,\ldots,L_{j-i}\in\Kn$.
\end{corollary}

\subsection*{Acknowledgments}
Parts of this project were carried out while the authors visited the Institute for Computational
and Experimental Research in Mathematics in Providence, RI, during the Harmonic Analysis and Convexity program in Fall 2022. Daniel Hug was supported by DFG research grant HU 1874/5-1
(SPP 2265). Fabian Mussnig was supported by the Austrian Science Fund (FWF): 10.55776/J4490 and 10.55776/P36210. Jacopo Ulivelli was supported by the Austrian Science Fund (FWF): 10.55776/P34446, and, in part, by the Gruppo Nazionale per l’Analisi Matematica, la Probabilit\'a e le loro Applicazioni (GNAMPA) of the Istituto Nazionale di Alta Matematica (INdAM). The first and the third author express their gratitude to the Hausdorff Research Institute for Mathematics in Bonn, Germany, where part of this work was finalized and presented while they were in-residence during Spring 2024 for the Dual Trimester Program: ``Synergies between modern probability, geometric analysis and stochastic geometry".

\bigskip\bigskip\bigskip
\parindent 0pt\footnotesize

\parbox[t]{8.5cm}{
Daniel Hug\\
Institut f\"ur Stochastik\\
Karlsruhe Institute of Technology (KIT)\\
Englerstra{\ss}e 2\\
76128 Karlsruhe, Germany\\
e-mail: daniel.hug@kit.edu}

\bigskip

\parbox[t]{8.5cm}{
Fabian Mussnig\\
Institut f\"ur Diskrete Mathematik und Geometrie\\
TU Wien\\
Wiedner Hauptstra{\ss}e 8-10/1046\\
1040 Wien, Austria\\
e-mail: fabian.mussnig@tuwien.ac.at}

\bigskip

\parbox[t]{8.5cm}{
Jacopo Ulivelli\\
Institut f\"ur Diskrete Mathematik und Geometrie\\
TU Wien\\
Wiedner Hauptstra{\ss}e 8-10/1046\\
1040 Wien, Austria\\
e-mail: jacopo.ulivelli@tuwien.ac.at}

\end{document}